\newfont{\bcb}{msbm10}
\newfont{\matb}{cmbx10}
\newfont{\got}{eufm10}

\documentclass[12pt]{amsart}
\usepackage{amsmath, amsthm, amscd, amsfonts, amssymb, latexsym, graphicx, color}
\usepackage[bookmarksnumbered, colorlinks, plainpages, hypertex]{hyperref}

\usepackage[cp1250]{inputenc}

\usepackage{amsmath,amsthm}
\usepackage{amssymb,latexsym}
\usepackage{enumerate}

\newtheorem{theorem}{Theorem}[section]

\newtheorem{proposition}[theorem]{Proposition}
\newtheorem{corollary}[theorem]{Corollary}

\theoremstyle{definition}

\newtheorem{example}[theorem]{Example}

\newtheorem{remark}[theorem]{Remark}
\numberwithin{equation}{section}

\begin{document}

\title[Pulling back singularities]{Pulling back singularities \\ for analytic complete intersections}

\author[Krzysztof Jan Nowak]{Krzysztof Jan Nowak}

%\footnotetext{Research partially supported by by KBN Grant No.\
%1P03A 00527.}



\date{}

\begin{abstract}
The following pullback problem will be considered. Given a finite holomorphic map germ $\phi : (\mathbb{C}^{n}, 0) \to (\mathbb{C}^{n}, 0)$ and an analytic germ $X$ in the target, if the preimage $Y = \phi^{-1}(X)$, taken with the reduced structure, is smooth, so is $X$. The main aim of this paper is to give an affirmative solution for $X$ being a geometric complete intersection. The case, where $Y$ is not contained in the ramification divisor $Z$ of $\phi$, was established by Ebenfelt--Rothschild (2007) and afterwards by Lebl (2008) and Denkowski (2016). The hypersurface case was achieved by Giraldo--Roeder (2020) and recently by Jelonek (2023).
\end{abstract}

\maketitle

\section{Introduction}

\vspace{1ex}

The main aim of this paper is to give an affirmative solution to the following pullback problem

\begin{theorem}\label{back}
Let $\phi : (\mathbb{C}^{n}, 0) \to (\mathbb{C}^{n}, 0)$ be a germ of a finite holomorphic map,
$(X,0) \subset (\mathbb{C}^{n}, 0)$ an analytic germ being a geometric complete intersection in the target, and $(Y,0)$ the preimage of $(X,0)$ under $\phi$, taken with the reduced structure. Then if $(Y,0)$ is smooth (non-singular), so is the germ $(X,0)$.
\end{theorem}

%\vspace{1ex}

This problem, for an arbitrary (irreducible) analytic germ $X$, was first posed by Ebenfelt--Rothschild~\cite[Section~4]{ER} in relation with their study of the images of CR submanifolds under finite holomorphic maps. They gave an affirmative answer in the case where the Jacobian of the map $\phi$ does not vanish on the analytic germ $Y$ (thus if $Y$ is not contained in the ramification divisor $Z$ of $\phi$) and where $(X,0)$ is a curve.

\vspace{1ex}

Afterwards some partial solutions were established in a different way by Lebl~\cite{L} and Denkowski~\cite{Den} under some extra assumptions, including that on the Jacobian of the map $\phi$.

\vspace{1ex}

For hypersurfaces (codimension one germs) this problem was eventually solved by Giraldo--Roeder~\cite{GR} (who analyzed the singular holomorphic foliation of codimension one associated to $X$, and some Koszul complex) and recently by Jelonek~\cite{Jel} (by means of Milnor numbers).

%\begin{remark}\label{Lebl}
%Note that Lebl solved a similar pullback problem for the property of being normal, which is of course equivalent to the one under study for the case of curves.
%\end{remark}

\vspace{1ex}

First we give a short proof of the special case where the set $Y$ is not contain in the ramification divisor $Z$. It relies on the following two subtle results:

\vspace{1ex}

\begin{em}
1) Suppose a local ring $B$ is a finite module over a regular local ring $A$. Then $B$ is Cohen--Macaulay iff it is a free (equivalently, flat) $A$-module. In particular, $B$ is a free $A$-module if it is regular.
\end{em}

\vspace{1ex}

\begin{em}
2) Suppose  a local ring $B$ is a faithfully flat over a local ring $A$. Then if $B$ is regular, so is $A$.
\end{em}

\vspace{1ex}

The former result, sometimes called miracle flatness, is attributed to Hironaka (1962, his unpublished master's thesis, cf.~\cite[Theorem~25.16, Appendix~A.2, Historical note]{Na}), or sometimes to Palamodov (1967, for the rings of convergent power series). It can be also obtained by means of the Auslander--Buchsbaum formula relating depth and projective dimension, thus relying on homological algebra (syzygies). The latter, in turn, follows from Serre's homological criterion for a ring to be regular, relying as well on homological methods, including Koszul complexes and functors Tor (see e.g.~\cite[Chapter~8, Theorem~51]{Ma}).

\vspace{2ex}

{\em Proof of the case where $Y \not \subset Z$.} The map
$$ \phi: (\mathbb{C}^{n}_{y},0) \to (\mathbb{C}^{n}_{x},0), \ \ x=(x_1,\ldots,x_n), \ y=(y_1,\ldots,y_n) $$
induces a finite local ring homomorphism
$$ \phi^{*}: A := \mathbb{C}\{ y \} \to B := \mathbb{C}\{ x \}. $$
By result~1), $B$ is a free $A$-module of a rank $d$. Let $\mathfrak{a} \subset A$ and $\mathfrak{b} \subset B$ be the ideals of the analytic set germs $X$ and $Y$, respectively. Then $A/ \mathfrak{a} \subset B/\mathfrak{a}B$ and $B/\mathfrak{a}B$ is a finite free module over $A/ \mathfrak{a}$ of rank $d$. If $Y \not \subset Z$, the map germ $\phi$ and its restriction $\phi|Y: Y \to X$ have the same multiplicity. Hence and by Chevalley's multiplicity formula, we get
$$ d = [B:A] = [B/\mathfrak{b} : A/\mathfrak{a}]; $$
here $[B:A]$ denotes the maximal number of elements from $B$ which are linearly independent over $A$. Since $B/\mathfrak{a}B$ is a free module over $A/\mathfrak{a}$ and $d = [B/\mathfrak{a}B : A/\mathfrak{a}]$, it follows that $\mathfrak{b} = \mathfrak{a}B$, and thus $B/\mathfrak{b}$ is a free module over $A/\mathfrak{a}$ too. By result~2, the local ring $A/\mathfrak{a}$ is regular too. This means that the analytic set germ $X$ is smooth, as desired.  \hspace*{\fill} $\Box$

\vspace{2ex}

Dealing with geometric complete intersections of an arbitrary codimension $p$, we shall apply the theory of singular complex analytic foliations. For its rudiments, we refer the reader e.g.\ to the paper~\cite{Yo}. What is crucial for our approach is that the singular foliation of codimension $p$ under study can be associated with a holomorphic $p$-form. Note also that an approach to singular analytic foliations via differential $p$-forms goes back to the paper~\cite{Med}.
%It should be also mentioned that the problems we encounter in the geometric complete %intersection case are much more difficult in comparison with the hypersurface %case.

\vspace{1ex}

Some essential ingredients of our approach are, among others, the following three results recalled below for the reader's convenience:

\vspace{1ex}

%1) a cohomological Riemann--Hartogs extension theorem (Scheja~\cite{Sche}, see also~\cite[Chapter~V, Section~4]{G-R});

%\vspace{1ex}

1) a local complex analytic version of Kleiman's theorem on a generic translate (cf.~\cite[Theorem~2]{Kle}, see also~\cite{Sp}), which allows us to proceed with induction;

\vspace{1ex}

2) Saito's generalized version of the theorem on Koszul cohomology (cf.~\cite{Sai}).

\vspace{1ex}

3) Continuity of intersections of analytic cycles, as developed by Tworzewski--Winiarski~\cite{T-W,T} and Chirka~\cite[\S 12 and \S 15]{Ch}.

\vspace{1ex}

Kleiman's theorem can be regarded as a generalization of the classical Bertini theorem.
His proof relies on some results from Grothendieck's EGA~\cite{Gro} formulated for general morphisms between locally noetherian schemes or for morphism locally of finite type.

\begin{theorem}\label{Kle} (Kleiman)
Consider an integral algebraic group scheme $G$ and an integral homogeneous $G$-space $X$ (i.e.\ an
algebraic scheme with a transitive $G$-action). Let $f: Y \to X$ and $g: Z \to X$
be two maps of integral algebraic schemes. For each rational point $s$ of $G$,
let $s Y$ denote $Y$ considered as an $X$-scheme via the map $y \mapsto sf(y)$.

1) Then, there exists a Zariski dense and open subset $U \subset G$ such that, for each rational
point $s \in U$, either the fibred product $(sY) \times_{X} Z$ is empty, or it is
equidimensional with the expected dimension;

2) Moreover, if the characteristic is zero, and $Y$ and $Z$ are regular (reduced), then there exists a Zariski dense and open subset $U \subset G$ such that, for each rational point $s \in U$, the fibred product $(sY) \times_{X} Z$ is regular (reduced), and the translate
map $(sf)$ is transversal to $g$.  \hspace*{\fill} $\Box$
\end{theorem}

Also available is a complex analytic counterpart, where $G$ is a complex algebraic group, $X$ is a complex homogeneous $G$-space, $f:Y \to X$ and $g: Z \to X$ are two finite maps of complex analytic spaces. The result in the complex analytic case can be established by means of some transversality arguments and Sard's theorem, which yields a subset $U$ as above such that $G \setminus U$ is a nowhere dense algebraic subset (of measure zero) of $G$.

\begin{remark}\label{Kle-an}
In the local complex analytic case one should take any non-empty, open (in the Euclidean topology) subset $G_{0}$ of $G$ such that every translate $(sY)$, $s \in G_{0}$, meets $Z$. Then the corresponding local analytic version yields a subset $U$ which is the complement in $G_{0}$ of a nowhere dense algebraic subset.
\end{remark}

\vspace{1ex}

Before stating the second result, we adopt some notation. Consider a commutative noetherian ring $R$ with unit, and a free $R$-module $\Omega$ of a finite rank $n$ with basis $e_1,\ldots,e_n$. Let $\omega_{1},\ldots,\omega_{k} \in \Omega$, $k \in \{ 1,\ldots,n \}$,
$$ \omega_{1} \wedge \ldots \wedge \omega_{k} = \sum_{1 \leq i_{1} < \ldots < i_{k} \leq n} \, a_{i_{1},\ldots,i_{k}} \, e_{i_{1}} \wedge \ldots \wedge e_{i_{k}}, $$
and let $\mathfrak{a}$ be the ideal generated by the coefficients
$$ a_{i_{1},\ldots,i_{k}}, \ \ \ 1 \leq i_{1} < \ldots < i_{k} \leq n. $$
For $p \in \{0,1,\ldots,n \}$, put
$$ Z^{p} := \left\{ \omega \in \bigwedge^{p} \Omega : \; \omega \wedge \omega_{1} \wedge \ldots \wedge \omega_{k} = 0 \right\} $$
and
$$ H^{p} := \frac{Z^{p}}{\sum_{i=1}^{k} \, \omega_{i} \wedge \bigwedge^{p-1} \, \Omega}. $$

\begin{theorem}\label{RS} (Saito)
Under the above assumptions, there exists a non-negative integer $m$ such that
$$ \mathfrak{a}^{m} H^{p} =0 \ \ \ \text{for} \ \ p=0,1, \ldots,n. $$
Further, $H^{p} =0$ for all $p < \mathrm{depth}_{\mathfrak{a}} (R)$.  \hspace*{\fill} $\Box$
\end{theorem}

\vspace{1ex}

Finally, we remind the reader the concept from~\cite{T-W} of a (metrizable) topology of local uniform convergence on the family $\mathrm{Cl}(M)$ of all closed subsets of a metric space $M$. It is applied to analytic sets and, along with intersection indices, to positive analytic cycles, being then equivalent to the convergence in the sense of currents. That topology is generated by the sets
$$ \mathcal{U}(\Omega,K) := \{ F \in \mathrm{Cl} (M): \ F \cap K = \emptyset, \ F \cap \Omega \neq \emptyset \} $$
for all compact subsets $K \subset M$ and all open subsets $\Omega \subset M$.

\vspace{1ex}

Suppose now that $X$ is a geometric complete intersection of codimension $p$ in an open subset $U \subset \mathbb{C}^{n}$, i.e.\ $X = X_{0}$ is the zero locus of some $p$ holomorphic functions $f_{i}: U \to \mathbb{C}$, $i=1,\ldots,p$ of dimension $n-p$. The analytic level sets
$$ X_{c} = \{ x \in U: \  f(x) = c \}, \ \ f = (f_{1},\ldots,f_{p}), \ \
   c=(c_{1},\ldots,c_{p}),$$
remain geometric complete intersections for $c$ from a neighbourhood $T$ of $0 \in \mathbb{C}^{p}$. Replace the domain $U$ by $f^{-1}(T)$. By the Remmert open mapping theorem (cf.~\cite[Chapter~V, \S~6, Theorem~2]{Loj}), the map $f$ is open. Hence $f(\Omega)$ is an open subset of $T$ for any open subset $\Omega$ of $U$. It is thus not difficult to deduce the following fact, which will be used in the sequel.

\begin{proposition}\label{level}
Suppose a sequence of parameters $c_{\nu} \in T$, $\nu \in \mathbb{N}$, $\nu \geq 1$, tends to a parameter $c_{0} \in T$. Then the sequence of level sets $X_{c_{\nu}}$ tends to $X_{c_{0}}$ in the topology of local uniform convergence.    \hspace*{\fill} $\Box$
\end{proposition}

\section{Singular complex analytic foliations}

\vspace{1ex}

Before turning to the main theorem, we recall some results about the singular set of a singular foliation from~\cite{Yo}, which will be used in its proof.
Denote by $\mathcal{O} =\mathcal{O}_{M}$, $\Theta = \Theta_{M}$, $\Omega = \Omega_{M}$ and $\Omega^{p} = \Omega_{M}^{p}$ the structure, tangent, cotangent sheaf and the sheaf of analytic $p$-forms on $M = \mathbb{C}^{n}$, respectively.
Singular complex analytic foliations can be treated both as integrable (involutive) coherent subsheaves $\mathcal{E}$ of the tangent sheaf $\Theta$ of vector fields, or integrable coherent subsheaves $\mathcal{F}$ of the cotangent sheaf $\Omega$ of 1-forms. Adopt the notation from~\cite{Yo} and denote by $\mathcal{E}^{a} \subset \Omega$ and $\mathcal{F}^{a} \subset \Theta$ the associated reduced foliations. By the reduction of the foliation $\mathcal{E}$ or $\mathcal{F}$ we mean the associated foliations $(\mathcal{E}^{a})^{a}$ or $(\mathcal{E}^{a})^{a}$, respectively.
Notice that a singular foliation is reduced iff it coincides with its reduction and iff its singular locus is at least of codimension $2$. The singular loci of a reduced singular foliation and the associated one coincide, so that one can equivalently investigate reduced foliations in terms of vector fields and in terms of 1-forms.
Henceforth we shall work only with reduced singular foliations.

\vspace{1ex}

The common singular loci of the foliations $\mathcal{E}$ and $\mathcal{F}$, related above, are the subsets where the corresponding sheaves are not locally free:
$$ \mathrm{Sing}\, (\mathcal{E}) = \{ x \in M: \ (\Omega_{M}/\mathcal{E})_{x} \ \ \text{is not} \ \ (\mathcal{O}_{M})_{x}\text{-free} \} $$
and
$$ \mathrm{Sing}\, (\mathcal{F}) = \{ x \in M: \ (\Omega_{M}/\mathcal{F})_{x} \ \ \text{is not} \ \ (\mathcal{O}_{M})_{x}\text{-free} \}. $$
Note that the rank of $\mathcal{E}$ or of $\mathcal{F}$ at a point $x \in M$, as defined in~\cite{Yo}, is equal $n$ minus the corank of $\Theta_{M}/\mathcal{E}$ or of $\Omega_{M}/\mathcal{F}$ at $x$, as defined in~\cite[Chapter~4, \S~4]{CAS}.
Clearly, these singular loci can be described by the formulae
$$ \{ x \in M: \ \dim_{\mathbb{C}} \, (\Omega_{M}/\mathcal{E})(x) > p \} = \{ x \in M: \ \dim_{\mathbb{C}} \, \mathcal{E}(x) < n-p \} $$
and
$$ \{ x \in M: \ \dim_{\mathbb{C}} \, (\Omega_{M}/\mathcal{F})(x) > p \} = \{ x \in M: \ \dim_{\mathbb{C}} \, \mathcal{F}(x) < n-p \}. $$
The generic (minimal) corank of $\mathcal{E}$ and $\mathcal{F}$ on $M$ is called their corank; and the generic (maximal) rank of $\mathcal{E}$ and $\mathcal{F}$ is called their rank (or dimension).

\vspace{1ex}

For each point $x \in M$, put
$$ \mathcal{E}(x) := \{ v(x): \ v \in \mathcal{E}_{x} \} $$
where $v(x)$ denotes the evaluation of the vector field germ $v$ at $x$; obviously, $\mathcal{E}(x)$ is a sub-vector space of the tangent space $T_{x}M$.

\vspace{1ex}

For $k \in \{ 0,1,\ldots, n-p \}$, put
$$ L^{(k)} := \{ x\ \in M: \ \dim_{\mathbb{C}}\, \mathcal{E}(x) = k \}, \ \  S^{(k)} := \{ x \in M: \ \dim_{\mathbb{C}} \mathcal{E}(x) \leq k \}, $$
and $L^{(-1)}=S^{(-1)}=\emptyset$ for convenience. We have
$$ L^{(k)} = S^{(k)} - S^{(k-1)}, \ \ S^{(k)}= \bigcup_{i=0}^{k}L^{(i)}; $$
moreover, all the above sets $S^{(k)}$ are analytic (cf.~\cite[Proposition~2.3]{Yo}). Whenever the foliation $\mathcal{E}$ is involutive, the canonical filtration
$$ \emptyset = S^{-1} \subset S^{0} \subset \ldots \subset S^{n-p-1} = \mathrm{Sing}\, (\mathcal{E}) \subset S^{n-p} = M $$
controls the tangent spaces $\mathcal{E}(x)$ of $\mathcal{E}$: $\mathcal{E}(x)$ is contained in the tangent cone of $S^{k}$ at $x$ if $x \in S^{(k)}$ (the tangency lemma, {\em loc.cit.}). Relying on the tangency lemma, one can achieve ({\em op.cit.}, Theorems~2.15 and~2.16) the existence of (singular) integral submanifolds and local analytical triviality along each (also singular) leaf.

\section{Reduced pullback of the induced singular foliation}

\vspace{1ex}

Let $X$ be a geometric complete intersection of codimension $p$ in an open subset $U \subset \mathbb{C}^{n}$ determined by some $p$ holomorphic functions
$$ f_{i}: U \to \mathbb{C}, \ \ i=1,\ldots,p. $$ Since we are interested in what happens near the origin, $U$ may be just a polydisk in $\mathbb{C}^{n}$.
The subsheaf of $U$ generated by the differentials $df_{1}, \ldots,df_{p}$ is a coherent sheaf, locally free of rank $p$ off a nowhere dense analytic subset of $U$ (cf.~\cite[Chapter~4, \S~4]{CAS}). It is integrable (cf.~\cite[Section~1]{Yo}), and thus is a singular complex analytic foliation on $U$, depending actually on the choice of holomorphic functions $f_{i}$. Then the reduction $\mathcal{F}_{X}$ of this foliation is induced by the reduction $\omega_{X}$ of the associated holomorphic $p$-form
$df_{1} \wedge \ldots \wedge df_{p}$.
Here the reduction of a holomorphic $p$-form $\omega$ is, by definition, a unique (up to a unit factor) holomorphic $p$-form obtained via dividing its coefficients by their common holomorphic factor.

%\begin{remark}
%Note that the image $f(\mathrm{Sing}\, (\mathcal{F}_{X}))$ is an analytic subset near the origin in $\mathbb{C}^{p}$. Actually, so are the images $f(S^{k})$, $k=0,1,\ldots,p$, %keeping the notation from Section~2. This can be proven by induction with respect to $k$ by means of the Remmert--Stein theorem on removable singularities (see %e.g.~\cite[Chapter~IV, \S~6]{Loj}).
%\end{remark}

\begin{remark}\label{dim}
In the sequel we may assume that $p \leq n-2$ since the case where $X$ is a curve has already been proven in~\cite{ER,L}.
\end{remark}

Turning to the pullback problem, denote by
$\mathcal{G} := \phi^{\#}(\mathcal{F}_{X})$
the reduced pullback of $\mathcal{F}_{X}$ under the finite map $\phi$, i.e.\ the reduction of the pullback $\phi^{*}(\mathcal{F}_{X})$ under $\phi$. The foliation $\mathcal{G}$ corresponds to a (unique up to a unit factor) holomorphic $p$-form $\omega_{Y}$ which is the reduction of the  pullback $\phi^{*}(\omega_{X})$.
Then the analytic level sets of the foliation $\mathcal{G}$:
$$ Y_{c} = \{ g = c \}, \ \ g = (g_{1},\ldots,g_{p}), \ c=(c_{1},\ldots,c_{p}), $$
where $g_{i} := f_{i} \circ \phi$, $i=1,\ldots,p$, are geometric complete intersections for parameters $c$ from the neighbourhood $T$ of $0 \in \mathbb{C}^{p}$, considered in Section~1.

\vspace{1ex}

%In Section~5, we shall still need the following

%\begin{lemma}\label{red-form}
%Suppose that $p \geq n-2$ and that the analytic germ $X$ has an isolated singularity. Then the holomorphic $p$-form $df_{1} \wedge \ldots \wedge df_{p}$ is reduced, and thus it %equals $\omega_{X}$ in the vicinity of the origin.
%\end{lemma}

%\begin{proof}
%Otherwise the $p$-form $df_{1} \wedge \ldots \wedge df_{p}$ would vanish on a hypersurface $S$ through the origin. Then the analytic germ $X \cap S$ at the origin, being the %singular locus of $X$, would be of positive dimension. But this contradicts the assumption on the germ $X$.
%\end{proof}

Now we prove a key result which allows us to replace local analysis of an analytic germ $X$ with that of the singular foliation $\mathcal{F}_{X}$. The idea behind the proof is to make use of continuity of the above level sets in the topology of local uniform convergence (Proposition~\ref{level}) and of continuity of analytic intersections.

\begin{proposition}\label{basic}
Adopt the foregoing notation and the assumptions of the pullback problem. Then the reduced pullback $\mathcal{G} = \phi^{\#}(\mathcal{F}_{X})$ is a non-singular analytic foliation in the vicinity of the origin.
\end{proposition}

\begin{proof}
Choose local coordinates $y_{1},\ldots, y_{n}$ in a small polydisk $D$ centered at the origin such that
$$ Y_{0} = \{ y_{1} = \ldots = y_{p} =0 \} $$ and assume that the foliation $\mathcal{G}$ is defined in $D$. Let
$$ \pi: \mathbb{C}^{n} \to \mathbb{C}^{n-p} $$
be the canonical projection onto the last $(n-p)$ coordinates.

\vspace{1ex}

We begin by showing that, for parameters $c$ close enough to $0 \in \mathbb{C}^{p}$, the level sets $Y_{c}$ are smooth and the restrictions $\pi|Y_{c}$ are injective. This follows from Proposition~\ref{level} about the continuous dependence of the level sets on parameter, and from continuity of analytic intersections (cf.~\cite[Corollary~12.3.4]{Ch} and \cite[Theorem~3.6]{T}).

\vspace{1ex}

Indeed, suppose first that there exist a sequence $c_{\nu}$ of parameters tending to the origin and a sequence $b_{\nu}$ of singular points of $Y_{c_{\nu}}$ tending to a point $b_{0} \in Y_{0}$. Take a vector space $L \subset \mathbb{C}^{n}$ of dimension $p$ such that each affine line $b_{\nu} + L$ and $Y_{c_{\nu}}$ intersect properly. By continuity of analytic intersections, the intersection index of the intersection of $b_{0} + L$ and $Y_{c_{0}}$ at $b_{0}$ coincides with that of $b_{\nu} + L$ and $Y_{c_{\nu}}$ at $b_{\nu}$ for $\nu$ large enough. We thus get a contradiction, since the first index equals $1$ and the others are greater than $1$. Hence the first conclusion holds.

\vspace{1ex}

Next suppose that there exist a sequence $c_{\nu}$ of parameters tending to the origin and a sequence $a_{\nu}$ of $Y_{0}$ tending to a point $a_{0} \in Y_{0}$ such that the intersections of $a_{\nu} + L$ and $Y_{c_{\nu}}$ are of cardinality greater than $1$, where
$$ L := \{ y \in \mathbb{C}^{n}: \ y_{p+1} = \ldots = y_{n} = 0 \}. $$
Similarly as before, we get a contradiction by continuity of analytic intersections. Hence the second conclusion holds.

\vspace{1ex}

Summing up, the level sets $Y_{c}$ are, for parameters $c$ closed enough to the origin, the graphs of some holomorphic functions
$$ \theta_{c}: Y_{0} \to \mathbb{C}^{p}_{y_{1},\ldots,y_{p}}. $$
After perhaps shrinking the polydisk $Y_{0}$, those functions tend in the topology of uniform convergence, together with their partial derivatives, to the zero function when $c$ tends to the origin. Hence the level sets
$Y_{c}$ are non-singular leaves of the foliation $\mathcal{G}$. Therefore $\mathcal{G}$ is a non-singular analytic  foliation in the vicinity of the origin, as desired.
\vspace{1ex}

\end{proof}

The converse being clear, we immediately obtain the following

\begin{corollary}\label{non-sing}
The preimage $(Y,0)$ of the germ $(X,0)$ is smooth at $0 \in \mathbb{C}^{n}_{y}$ iff the reduced pullback $\mathcal{G} = \phi^{\#}(\mathcal{F}_{X})$ is a non-singular foliation near $0 \in \mathbb{C}^{n}_{y}$.   \hspace*{\fill} $\Box$
\end{corollary}

The following example indicates that in the above proposition also essential is the assumption that $\mathcal{G}$ is the reduced pullback $\phi^{\#}(\mathcal{F}_{X})$.

\begin{example}\label{ex-Y}
Consider the following family of analytic level sets $Y$ at $0 \in \mathbb{C}^{3}$:
$$ Y := \{ y \in \mathbb{C}^{3}: \ g_{1}(y) = y_{1}=c_{1}, \ g_{2}(y) = y_{3}^{3} - y_{2}^{2}\,  y_{1} = c_{2} \} $$
and the induced singular foliation $\mathcal{G}$ generated by the analytic 1-forms:
$$ dg_{1} = dy_{1}, \ \ dg_{2}= 3 y_{3}^{2}\, dy_{3} - 2 y_{2} y_{1}\, dy_{2} - y_{1}^{2} \, dy_{1}. $$
Then the singular locus of $\mathcal{G}$ is the union of two axes $Oy_{1}$ and $Oy_{2}$, thus of codimension 2, though the (smooth) level set $Y_{0}$ is the axis $Oy_{2}$.  \hspace*{\fill} $\Box$
\end{example}

\section{Proof of the main theorem}

\vspace{1ex}

We shall prove Theorem~\ref{back} by reductio ad absurdum, assuming that there exists a singular complex analytic germ $(X,0)$ in $\mathbb{C}^{n}$ being a geometric complete intersection, whose preimage $(Y,0)$, taken with the reduced structure, is smooth. Moreover, we shall proceed with induction on the dimension of the ambient space, by lowering that dimension until the induced foliation $\mathcal{F}_{X}$ has an isolated singularity. In that eventual case, a direct proof will be given.

\vspace{1ex}

To this end, we shall apply the local complex analytic version of Kleiman's theorem from Remark~\ref{Kle-an} to the general affine group $G$, which is the semidirect product
$$ G = \mathrm{Aff}(n,\mathbb{C}) = \mathrm{GL}(n,\mathbb{C}) \rtimes \mathbb{C}^{n} $$
of the additive group $\mathbb{C}^{n}$ of translations of $\mathbb{C}^{n}$ and the general linear group $\mathrm{GL}(n,\mathbb{C})$. The group $G$ acts transitively on $\mathbb{C}^{n}$ so that
$$ \mathbb{C}^{n} = \mathbb{A}^{n}(\mathbb{C}) \cong \mathrm{Aff}(n,\mathbb{C})/\mathrm{GL}(n,\mathbb{C}). $$
Also applied will be
R\"{u}ckert's descriptive lemma (cf.~\cite[Chapter~III, \S~3]{Loj}) and a transversality argument (similar in fact to the one behind the parametric transversality theorem; see e.g.~\cite[Chapter~2, \S~3]{G-P}).

%the local complex analytic version of Bertini's theorem (\cite[Theorem~3.3.1]{Laz}) which asserts the irreducibility of the generic intersections,

\vspace{1ex}

Now suppose that $\dim \, \mathrm{Sing}\, (\mathcal{F}_{X}) \geq 1$. We encounter two cases:

I. either $\dim \, \mathrm{Sing}\, (X) \geq 1$; or

II. the level set $X$ is an isolated complete intersection singularity.

\vspace{1ex}

In case~I, we can find by means of the above mentioned analytic version of Kleiman's, via routine dimension calculus, a nowhere dense analytic subset $V$ of $\mathrm{Sing}\, (X)$ such that the following holds. For each point
$$ a \in \mathrm{Sing}\, (X) \setminus V $$
there exists a Zariski open and dense subset $\Lambda (a)$ of the set of all linear hyperplanes in $\mathbb{C}^{n}$ such that for every $H \in \Lambda (a)$ the intersection
$X \cap (H + a)$ is a singular, irreducible, reduced, geometric complete intersection germ at $a$, and the preimage
$$  \widetilde{H}(a) := \phi^{-1}(H + a) \ \ \text{and} \ \ Y \cap \widetilde{H}(a) = \phi^{-1}(X \cap (H + a)) $$
is smooth; the irreducibility of the intersection follows immediately from that its preimage is smooth. Observe here that, having disposed of the irreducibility at $a$ of the intersection $X \cap (H +a)$, its singularity can be ensured by the generic condition for the affine hyperplane $H + a$ to contain a generic direction of the primitive element from R\"{u}ckert's descriptive lemma for the germ $X$ at $a$.

\vspace{1ex}

In case~II, the germ $X$ at $0 \in \mathbb{C}^{n}_{x}$ has an isolated singularity, and thus the map $f$ defines an isolated complete intersection singularity at $0 \in \mathbb{C}^{n}_{x}$ in the sense of~\cite[Section~1.B]{Loo}. By~\cite[Theorem~2.8]{Loo}, the map $f$ defines an isolated complete intersection singularity (or is non-singular) at every point $a$ in the vicinity of the origin.

\vspace{1ex}

Observe first that then the (even scheme-theoretical) intersection of the level set $X_{c}$, where $c=f(a)$, with a generic affine hyperplane $H +a$ is an isolated complete intersection singularity at $a$; and a fortiori it is reduced ({\em op.cit.}, Proposition~1.10). We outline the proof for $c=0$. Consider a linear projection $\vartheta: \mathbb{C}^{n}_{x} \to \mathbb{C}^{2}$ and its restriction
$$ \vartheta_{0}: X_{0} \setminus \{ 0 \}  \to \mathbb{C}^{2}. $$
By Sard's theorem (cf.~\cite[Chapter~V, \S~1, Corollary~2]{Loj}), the set of critical values of $\vartheta_{0}$ is a subanalytic subset of $\mathbb{C}^{2}$ and a countable union of complex analytic manifolds of complex dimension $<2$. Hence its intersection with a generic line $L$ through $0 \in \mathbb{C}^{2}$ is a subanalytic subset and a complex analytic set of complex dimension $<1$, and thus it is a finite set. Then the hyperplane $H := \vartheta^{-1}(L)$ is one we are looking for.

\vspace{1ex}

As before, we can find by means of the analytic versions of Kleiman's theorem, via routine dimension calculus, a nowhere dense analytic subset $V$ of $\mathrm{Sing}\, (\mathcal{F}_{X})$ such that the following holds. For each point
$$ a \in \mathrm{Sing}\, (\mathcal{F}_{X}) \setminus V $$
there exists a Zariski open and dense subset $\Lambda (a)$ of the set of all linear hyperplanes in $\mathbb{C}^{n}$ such that for every $H \in \Lambda (a)$ the preimage
$$  \widetilde{H}(a) := \phi^{-1}(H + a) $$
is a smooth analytic hypersurface. Take a point $a \in \mathrm{Sing}\, (\mathcal{F}_{X}) \setminus V$ and a point $b$ with $\phi(b)=a$. Clearly, the set of tangent spaces
$$ \left\{ T_{b}\widetilde{H} \subset \mathbb{C}^{n}: \ H \in \Lambda (a) \right\} $$
is a Zariski open and dense subset of the set of all linear hyperplanes in $\mathbb{C}^{n}$. Therefore some (even generic) of those tangent spaces, say $T_{b}\widetilde{H}$ with $H \in \Lambda^{\prime}(a)$, meet transversally the linear subspace
$$ \{ y_1 = \ldots = y_p = 0 \}. $$
Consequently, the preimages $\widetilde{H}(a)$ with $H \in \Lambda^{\prime} (a)$ meet transversally at $b$ the level set $Y_{c} = \phi^{-1}(X_{c})$ through $b$, where $c = g(b) = f(a)$. Therefore the germs at $b$ of the preimages
$$ \phi^{-1}(X_{c} \cap (H + a)) = \phi^{-1}(X_{c}) \cap \phi^{-1}(H + a) = Y_{c} \cap \widetilde{H}(a) \subset \widetilde{H}(a) $$
are smooth germs at $b$ both in $Y_{c}$ and in $\widetilde{H}(a)$ for all $H \in \Lambda^{\prime}(a)$.

\vspace{1ex}

Hence the intersections $X_{c} \cap (H + a)$ are irreducible germs at $a$ for all $H \in \Lambda^{\prime}(a)$.
On the other hand, it follows immediately from Proposition~\ref{basic} (applied to the foliation $\mathcal{F}_{X}$) that, given any point $a \in \mathrm{Sing}\, (\mathcal{F}_{X})$, the analytic level $X_{c}$, $c =f(a)$, is singular at $a$. Therefore again, the singularity of the intersection germs $X_{c} \cap (H + a)$ can be ensured by the generic condition for the hyperplane $H$ to contain a generic direction of the primitive element from R\"{u}ckert's descriptive lemma for the germ $X_{c}$ at $a$.

\vspace{1ex}

Summing up, we can assume that the germs $X_{c} \cap (H + a)$ at $a$, as above, are also singular, irreducible, geometric complete intersection germs at $a$.
Under the circumstances, the pullback problem under study comes down to the case where the new ambient space $H$ is of lower dimension $(n-1)$.
This allows us to proceed with induction on the dimension of the ambient space until the foliation $\mathcal{F}_{X}$ has an isolated singularity. In this manner, we are reduced to the following special case of the main theorem

\begin{proposition}\label{main-red}
Theorem~\ref{back} holds when the induced foliation $\mathcal{F}_{X}$ has an isolated singularity.
\end{proposition}

\section{Proof of Proposition~\ref{main-red}}

\vspace{1ex}

Again we proceed with reductio ad absurdum. Suppose thus that $X$ is a geometric complete intersection at $0 \in \mathbb{C}^{n}$ of codimension $p \leq n-2$ (cf.~Remark~\ref{dim}), that its preimage $Y = \phi^{-1}(X)$ (with the reduced structure) is a smooth analytic germ, and the induced foliation $\mathcal{F}_{X}$ has an isolated singularity. This foliation is determined in a neighbourhood $U$ of the origin by the reduction $\omega_{X}$ of the associated differential $p$-form
$df_{1} \wedge \ldots \wedge df_{p}$.

\vspace{1ex}

Proposition~\ref{basic} states that the reduced pullback $\mathcal{G} = \phi^{\#}(\mathcal{F}_{X})$ is a non-singular foliation, and thus is determined by an associa\-ted differential $p$-form $\omega_{Y}$ which is in some local coordinates $y$ of the form
$$ \omega_{Y} = dy_{1} \wedge \ldots \wedge dy_{p}. $$
As noted in Section~3, the $p$-form $\omega_{Y}$ is the reduction of the pullback
$$ \phi^{*}\omega_{X} = \phi^{*}(df_{1} \wedge \ldots \wedge df_{p}). $$
Following an idea of Giraldo--Roeder~\cite{GR}, we shall consider some differential forms $\eta$ and $\tau_{J}$ defined below.
Write
$$ \omega_{X} = \sum_{I} \, a_{I}(x) \, dx_{I}, \ \ dx_{I} = dx_{i_1} \wedge \ldots \wedge dx_{i_p}, $$
where
$$ I = (i_1,\ldots,i_p), \ \ i_1 < \ldots < i_p, \ \ \{i_1, \ldots, i_p \} \subset \{ 1,\ldots,n \}. $$
Then
$$ \phi^{*}\omega = \sum_{I} \sum_{J} \, a_{I}(\phi(y)) \, \frac{\partial \phi_{I}}{\partial y_{J}} \, dy_{J} = \sum_{J} \, b_{J}(y) \, dy_{J},  $$
where
$$ J = (j_1,\ldots,j_p), \ \ j_1 < \ldots < j_p, \ \ \{j_1, \ldots, j_p \} \subset \{ 1,\ldots,n \} $$
and
$$ b_{J}(y) = \sum_{I} \, a_{I}(\phi(y)) \, \frac{\partial \phi_{I}}{\partial y_{J}}. $$
Then we of course have $b_{J}(y) \equiv 0$ for all $J \neq J_{0} := ( 1,\ldots,p )$. Put
$$ \eta := \Sigma_{I} \, a_{I}(\phi(y)) \, dy_{I}  \ \ \ \text{and} \ \ \ \tau_{J} := \sum_{I} \, \pm \, \frac{\partial \phi_{I}}{\partial y_{J}} \, dy_{\widehat{I}}\, , $$
choosing an appropriate sign so as to get
$$ \tau_{J} \wedge \eta = 0 \ \ \ \text{for all} \ \ \ J \neq J_{0}; $$
here $\widehat{I}$ is the sequence of elements $\{ 1,\ldots,n \} \setminus I$ in ascending order. Clearly, we have
$$ \eta = \widetilde{df_{1}} \wedge \ldots \wedge \widetilde{df_{p}}, $$
where the 1-forms $\widetilde{df_{1}},\ldots,\widetilde{df_{p}}$ are obtained from the 1-forms $df_{1},\ldots,df_{p}$ by substitution $\phi(y), dy_{1},\ldots,dy_{n}$ for $x,dx_{1},\ldots,dx_{n}$.

\vspace{1ex}

Denote by $\mathfrak{a}$ the ideal generated by the coefficients of the $p$-form $\eta$. Since
$$ \mathrm{Sing}\, (\mathcal{F}_{X}) = \{ x \in U: \omega(x) =0 \} = \{ 0 \} \subset \mathbb{C}^{n}_{x} $$
the zero locus of the ideal $\mathfrak{a}$ is the singleton $\{ 0 \} \subset \mathbb{C}^{n}_{y}$.

\vspace{1ex}

It follows directly from Theorem~\ref{RS} that
$$ \mathfrak{a}^{m} \cdot \tau_{J} \in \sum_{i=1}^{p} \, \widetilde{df_{i}} \wedge \Omega^{n-p-1} \ \ \ \text{for all} \ \ \ J \neq J_{0} $$
and for some non-negative integer $m$.
Hence it follows immediately that
$$ y_{s}^{N} \cdot \tau_{J} \in \sum_{i=1}^{p} \, \widetilde{df_{i}} \wedge \Omega^{n-p-1} \subset \sum_{j=1}^{n} \phi_{j} \cdot \Omega^{n-p} $$
for all $s \in \{ 1,\ldots,n \}$, $J \neq J_{0}$ and for some non-negative integer $N$.

\vspace{1ex}

To complete the proof, we shall show that the belonging
\begin{equation}\label{belonging}
  y_{s}^{N} \cdot \tau_{J} \in \sum_{j=1}^{n} \phi_{j} \cdot \Omega^{n-p}, \ \ \ k=1,\ldots,n,
\end{equation}
leads to a contradiction. To this end, we still need the following

\begin{proposition}\label{key}
Given a finite holomorphic map germ
$$ \psi : (\mathbb{C}^{p}_{y}, 0) \to (\mathbb{C}^{p}, 0), $$
the Jacobian determinant
$$ J := \frac{\partial (\psi_{1},\ldots,\psi_{p})}{\partial
   (y_{1},\ldots,y_{p})} \not \in \sum_{i=1}^{p} \, \psi_{i} \, \mathcal{O}
   % = \sum_{i=1}^{p} \, \psi_{i} \, \mathbb{C}\{ y \}
$$
does not belong to the ideal generated by the components of the map $\psi$.    \hspace*{\fill} $\Box$
\end{proposition}

For the proof of the above proposition, we refer the reader to the book~\cite{A-GZ-V}, Theorem~1 from Section~11 of Chapter~5, which was written by A.G.~Hovansky. His proof is strictly connected with the concept of the trace under a finite map of a differential form, which goes back to Kunz's theory of K\"{a}hler differentials (cf.~\cite{Kunz}).  A purely algebraic counterpart of this subtle result can be found in~\cite{Zan}. Nevertheless, its proof is to some extent of geometric nature, leading via some commutative algebra to the generic case where the algebraic objects under study are in a sense universal for the problem.

\vspace{1ex}

So suppose now that formula~\ref{belonging} holds. This means that
$$ y_{s}^{N} \cdot \frac{\partial \phi_{I}}{\partial y_{J}}  \in \sum_{i=1}^{n} \, \phi_{i} \, \mathcal{O} $$
for all $I$, $J \neq J_{0}$ and $s \in \{ 1,\ldots,n \}$. Since we have assumed at the start of the proof that $p \leq n-1$, we may take for instance:
$$ I := \{ 1,\ldots,p \}, \ J := \{ 2,\ldots, p+1 \}, \ s := p+2, $$
and put
$$ \psi_{i}(y_{2},\ldots,y_{p+1}) := \phi_{i}(0,y_{2},\ldots,y_{p+1},a,0,\ldots,0). $$
Then $\psi := (\psi_{1},\ldots,\psi_{p})$ is a finite holomorphic map in the vicinity of the origin, and further we get
$$  a^{N} \cdot \frac{\partial (\psi_{1},\ldots,\psi_{p})}{\partial (y_{2},\ldots,y_{p+1})} = a^{N} \cdot \frac{\partial \phi_{I}}{\partial y_{J}}(0,y_{2},\ldots,y_{p+1},a,0,\ldots,0) $$
$$ \in \sum_{i=2}^{p+1} \, \psi_{i} \, \mathbb{C} \{ y_{2},\ldots,y_{p+1} \},  $$
for $a \in \mathbb{C} \setminus \{0\}$ close to the origin.
This contradicts Proposition~\ref{key}, completing the proof of Proposition~\ref{main-red}, and thus of the main theorem as well.

\vspace{5ex}

\begin{small}
%\begin{sc}
Institute of Mathematics

Faculty of Mathematics and Computer Science

Jagiellonian University

ul.~Profesora S.\ \L{}ojasiewicza 6

30-348 Krak\'{o}w, Poland

{\em E-mail address: nowak@im.uj.edu.pl}
%\end{sc}
\end{small}

\end{document}